\documentclass[12pt]{amsart} 
\usepackage{amssymb}
\usepackage{latexsym}
\usepackage{amsmath}
\usepackage{amsthm}
\usepackage{amsfonts}

\usepackage{soul}
\usepackage[usenames,dvipsnames]{xcolor}

\usepackage{doi}

\usepackage{mathabx}

\newtheorem{theorem}{Theorem}
\newtheorem{lemma}[theorem]{Lemma}
\newtheorem{corollary}[theorem]{Corollary}

\theoremstyle{definition}

\newtheorem*{acknowledgements*}{Acknowledgements}
\theoremstyle{remark}
\newtheorem{remark}[theorem]{Remark}
\newtheorem{note}[theorem]{Note}

\usepackage{color}

\newcommand{\R}{\mathbb R}

\newcommand{\C}{\mathbb C}
\newcommand{\T}{\mathbb T}

\newcommand{\D}{\mathbb D}
\def\he{H(E)}
\def\hf{H(\Lambda(\varphi))}
\def\laf{\Lambda(\varphi)}
\def\h{H^1(\D)}

\newcounter{obs}

\author[J.~Carrillo-Alan\'{\i}s]{Javier Carrillo-Alan\'{\i}s}
\address{Facultad de Matem\'aticas,
Universidad de Sevilla, 
Calle Tarfia s/n,  Sevilla 41012, Spain}
\email{fcarrillo@us.es}

\author[G.~P. Curbera]{Guillermo P. Curbera}
\address{Facultad de Matem\'aticas \& IMUS,
Universidad de Sevilla, 
Calle Tarfia s/n,  Sevilla 41012, Spain}
\email{curbera@us.es}

\title
[Extreme points of the unit ball of Hardy-Lorentz spaces]
{A note on extreme points of the unit ball \\ of Hardy-Lorentz spaces}

\begin{document}

\begin{abstract}
We show that inner functions are extreme points of the  unit ball of the Hardy-Lorentz
space $H(\Lambda(\varphi))$, for $\Lambda(\varphi)$ 
a Lorentz space with $\varphi$ strictly increasing and strictly concave.
\end{abstract}

\maketitle


\section{Introduction}
\label{S1}


In 1958 De Leeuw and Rudin characterized the extreme points of the unit ball
of the Hardy space $\h$, \cite[Theorem 1]{deleeuw-rudin}.

\begin{theorem}[De Leeuw-Rudin]\label{t-1}
A function in $\h$ is an extreme point of the unit ball of ${\h}$ if 
and only if has norm one and  is an outer function.
\end{theorem}

In 1974 Bryskin and Sedaev proved a partial extension of the  result
by De Leeuw and Rudin, \cite[Theorem 1]{bryskin-sedaev}.

Let  $E$ be a symmetric (rearrangement invariant) space on $\T$ satisfying 
the Fatou property.
The space $\he$ is the set of all analytic functions $f\colon\D\to\C$ such that
\begin{equation*}
\|f\|_{\he}:=\sup_{0\le r<1}\|f_r\|_E<\infty,
\end{equation*}
where $f_r(e^{i\theta}):=f(re^{i\theta})$, for $0\le r<1$ and 
$\theta\in[0,2\pi]$. Since $\he\subseteq\h$, functions in $f\in\he$ have a.e.\ boundary 
values, which we  still denote by $f$. Thus, for $f\in\he$, the function $e^{i\theta}\in\T\mapsto 
f(e^{i\theta})$ is defined a.e.\ and belongs to $E$.

The  space   $E$ has  
\textit{strictly monotone norm} if $\|f\|_E<\|g\|_E$ whenever  
$f,g\in E$ satisfy $|f(e^{i\theta})|\le|g(e^{i\theta})|$ a.e.\ in 
$\T$ and 
\begin{equation*}
m(\{e^{i\theta}\in\T:|f(e^{i\theta})|<|g(e^{i\theta})|\})>0.
\end{equation*}

\begin{theorem}[Bryskin-Sedaev]\label{t-2}
If  $E$ has strictly monotone norm,  then every outer function 
in $\he$  with norm one   is an extreme point of the unit ball of ${\he}$. 
\end{theorem}

For a function $\varphi\colon[0,2\pi]\to\R$, increasing and concave with 
$\varphi(0)=0$, the 
Lorentz
space $\laf$ is the set of all measurable functions $f\colon\T\to\C$ such that
\begin{equation*}
\|f\|_{\laf}:=\int_0^{2\pi} f^*(t)\varphi'(t)\,dt<\infty,
\end{equation*}
where $f^*\colon[0,2\pi]\to\R$ is the decreasing rearrangement of $f$.
The space $\Lambda(\varphi)$ is  symmetric. The Hardy-Lorentz space is $H(\Lambda(\varphi))$.
Since $\varphi$ being strictly monotone implies that  the norm in $\laf$ is strictly monotone, the following results holds, 
\cite[Corollary]{bryskin-sedaev}.

\begin{theorem}[Bryskin-Sedaev]\label{t-3}
Let $\varphi$ be a strictly increasing, concave  function on $[0,2\pi]$
with $\varphi(0)=0$.  
Then every outer function 
in $\hf$  with norm one   is an extreme point of the unit ball of ${\hf}$. 
\end{theorem}

A further result of Bryskin and Sedaev shows that not all norm one functions in   
$\hf$ are extreme points of the unit ball of  $\hf$, 
\cite[Theorem 2]{bryskin-sedaev}.

\begin{theorem}[Bryskin-Sedaev]\label{t-4}
Let $f $ be a norm 
one function in the space $\hf$ with  $f(\alpha)=0$ for some $\alpha\in\D$.
Set $\mu(\theta):=|f(e^{i\theta})|$, for $\theta \in [0,2\pi]$. Suppose 
that $\mu$ is continuously differentiable on
$[0,2 \pi]$ and for some $\delta>0$  the 
inequality $\mu'(\theta) \leq -\delta$ holds for  all $\theta \in [0,2\pi]$. 
Then $f$ is not an extreme point of the unit ball in $\hf$.
\end{theorem}

In 1978, E.~M.~Semenov stated the problem of describing the set of extreme
points of the unit ball of $\hf$, \cite{semenov}. 
The problem was part of a collection of 99 problems in linear and complex analysis
from the  Leningrad branch of the Steklov Mathematical Institute,  \cite[5.1 pp. 23--24]{khavin-1978},
and appeared in the subsequent list of problems in 1984, 
\cite[1.6 pp. 22--23]{khavin-1984}, and in 1994,
\cite[1.5 p. 12]{khavin-1994}.

For matters related to symmetric (rearrangement invariant) spaces we refer
the reader to \cite{bennett-sharpley}, \cite{krein-petunin-semenov}. For matters 
related to Hardy spaces we refer the reader to \cite{duren}.


\section{A result}
\label{S2}


It could be expected that the extreme points of the unit ball of
$\hf$ would behave as for $\h$, so that,
following Theorem  \ref{t-1}, every extreme point 
of the unit ball of $\hf$ is an outer function. We  show 
that is not necessarily the case.

\begin{theorem}\label{t-5}
Let $\varphi$ be an  increasing,  concave  function on $[0,2\pi]$ with 
$\varphi(0)=0$. Suppose  that 
if $f, g \in\laf$ satisfy 
$$
 \|f\|_{\laf}+ \|g\|_{\laf}=\|f+g\|_{\laf},
$$ 
then
$$
f^* + g^*=(f+g)^*.
$$
Then every inner  function is an extreme point of the unit ball of ${\hf}$. 
\end{theorem}

\begin{proof}
We  assume that the norm in $\laf$
is normalized so that $\|\chi_\T\|_{\laf}=1$. Let $f$ be an inner function.  
Since $|f(e^{i \theta})|=1$, a.e.\ in $\T$, we have  $f^*=\chi_{[0,2\pi]}$, so 
$f\in \hf$ and  $\|f\|_{\hf}=1$. 
 
In order to check that $f$  is an extreme point
of the unit ball of $\hf$ we need to show that  if for some $h\in  \hf$ we have
\begin{equation}\label{2}
\|f+h\|_{\hf}=\|f-h\|_{\hf}=\|f\|_{\hf}=1,
\end{equation}
then necessarily $h=0$. 

From \eqref{2} it follows that
\begin{equation*}\label{3}
\|f+h\|_{\hf}+\|f-h\|_{\hf}=2=\|2f\|_{\hf}.
\end{equation*}
It follows,   from the assumption, that
$$
 (f+h)^* + (f-h)^*= 2f^*=2\chi_{[0,2\pi]}=2.
$$
Since $(f+g)^*$ and $(f-g)^*$ are decreasing functions and their sum is a 
constant function, both functions $(f+g)^*$ and $(f-g)^*$ must be constant. 
Hence, $(f+h)^* = c_1$ and $(f-h)^*=c_2$. This, 
together with $\|f + h \|_{\hf}=\|f-h\|_{\hf}=1$, 
implies that $c_1=c_2=1$.
Thus, 
$$
|f(e^{i \theta})+h(e^{i \theta})|=1, \quad |f(e^{i \theta})-h(e^{i \theta})|=1,
\quad  \mathrm{a.e. }\; e^{i \theta}\in\T. 
$$
Thus,  we have
$$
|f(e^{i  \theta}) + h(e^{i \theta})|= |f(e^{i  \theta}) - h(e^{i \theta})|,
\quad  \mathrm{a.e. }\; e^{i \theta}\in\T.
$$
This is only possible if $h(e^{i \theta})=0$ a.e. Thus,  $h=0$ on $\D$, and so $f$ 
is an extreme point of the unit ball of $\hf$.
\end{proof}


The following result gives a simple and explicit condition on the function 
$\varphi$ for the space 
$\laf$ satisfying the requirement of Theorem \ref{t-5}.

\begin{lemma}\label{l-3}
Let $\varphi$ be a strictly increasing and  strictly concave  
function on $[0,2\pi]$ with 
$\varphi(0)=0$. Suppose that  $f,g \in \Lambda(\varphi)$ satisfy 
$$\|f\|_{\Lambda(\varphi)} + \|g\|_{\Lambda(\varphi)}=\|f+g\|_{\Lambda(\varphi)},
$$
then 
$$ f^* + g^*=(f+g)^*. 
$$
\end{lemma}

\begin{proof}
We use the following  alternative expression for the norm  $\|f\|_{\laf}$
of $f$ in $\laf$:
\begin{equation}\label{1}
\|f\|_{\laf} =  \int_0^{2\pi} F(t)  d(- \varphi'(t)) 
+ \varphi'(2\pi)  \|f\|_{L^1(\T)},
\end{equation}
where 
$$
F(t):=\int_0^t f^*(s) ds, \quad t\in [0,2 \pi],
$$
which follows, via integration-by-parts, from the definition of the norm in $\laf$;
see, for example, \cite{carothers-turett}.

From 
$$
\|f\|_{\laf} + \|g\|_{\laf} = \|f+g\|_{\laf} \leq \| |f| + |g| \|_{\laf} \leq \|f\|_{\laf} + \|g\|_{\laf},
$$
it follows that  
$$
\|f+g\|_{\laf} = \| |f| + |g| \|_{\laf}.
$$
Since $\varphi$ is strictly increasing, the norm in $\laf$ is strictly monotone.
It follows that $|f+g| = |f| + |g|$ a.e.\ in $\T$. This implies (is equivalent to, in fact) that
$f \cdot g \geq 0$ a.e.\ in $\T$ and  $\|f+g\|_{L^1(\T)} = \|f\|_{L^1(\T)} + \|g\|_{L^1(\T)}$.

Consider, for $t\in [0,2 \pi]$, the functions
$$
F_1(t):=\int_0^t (f+g)^*(s)\, ds, \qquad F_2(t):=\int_0^t \left(f^*(s) +g^*(s) \right)\, ds.
$$
We will prove that $F_1 = F_2$. Since $\|f\|_{L^1(\T)} + \|g\|_{L^1(\T)} =\|f+g\|_{L^1(\T)}$, we have, via \eqref{1}, that
\begin{equation*}
\begin{split}
0 & =  \|f\|_{\laf} + \|g\|_{\laf} - \|f+g\|_{\laf}  \\ 
& = \int_0^{2 \pi}(F_2(t) - F_1(t))\, d( - \varphi'(t) ) \\
& \qquad + \varphi'(2 \pi)\Big( \|f\|_{L^1(\T)} + \|g\|_{L^1(\T)} -\|f+g\|_{L^1(\T)} \Big) \\
& = \int_0^{2 \pi}(F_2(t) - F_1(t))\, d( - \varphi'(t) ).
\end{split}
\end{equation*}
The function $\varphi$ being strictly concave implies that 
$-\varphi'$ is strictly increasing, so that $d( - \varphi'(t) )>0$. 
Since $F_2\ge F_1$, it follows that   $F_1 = F_2$.  
We deduce  that  $(f+g)^* = f^* + g^*$.
\end{proof}


Theorem \ref{t-5} together with Lemma \ref{l-3} yield the following result.

\begin{corollary}\label{cor}
Let $\varphi$ be a strictly increasing and strictly  concave  function on $[0,2\pi]$ with 
$\varphi(0)=0$. Then every inner  function is an extreme point of the unit ball of ${\hf}$. 
\end{corollary}

\begin{remark}
(i) For the Lorentz space $L^{p,1}$, with $1 < p < \infty$, 
Carothers and Turett have shown that if $f, g \in L^{p,1}$ satisfy 
$\|f\|_{p,1}+ \|g\|_{p,1} =\|f+g\|_{p,1}$, then $f^* + g^*=(f+g)^*$,
\cite[Lemma 1]{carothers-turett}. Accordingly, from Theorem \ref{t-5} it follows that 
every inner function is an extreme point of the unit ball of $H(L^{p,1})$.

(ii) The requirement  of $\varphi(t)$ being  strictly concave can not be dropped
in Corollary \ref{cor}. 
To see this, consider $\varphi(t)=t$, in which case $\laf=L^1(\T)$ so
${\hf}=\h$, where inner functions are not extreme points.
\end{remark}

\begin{acknowledgements*}
The authors thank Sergey V. Astashkin for bringing this problem to their attention.
\end{acknowledgements*}



\end{document}